\documentclass[12pt]{amsart}
\usepackage{amsmath}
\usepackage{amsfonts}
\usepackage{amssymb}
\usepackage{amscd}
\usepackage{mathtools}
\usepackage{amsxtra}
\usepackage{amsthm}
\usepackage{graphicx}
\usepackage[abbrev,alphabetic]{amsrefs}
\RequirePackage[dvipsnames,usenames]{color}
\usepackage{soul,xcolor}
\setstcolor{red}
\usepackage{stmaryrd}
\usepackage{booktabs}
\usepackage{multirow}
\newtagform{tiny}{\tiny(}{)}
\usepackage{threeparttable}

\usepackage{mathtools}
\usepackage{hyperref}
\usepackage[margin=1.25in]{geometry}
\usepackage{mathrsfs}

\usepackage{amsthm}
\usepackage{comment}
\usepackage[all,cmtip]{xy}
\usepackage{tikz-cd}
\usetikzlibrary{cd}

\tikzcdset{
  cells={font=\everymath\expandafter{\the\everymath\displaystyle}},
}

\usepackage[all]{xy}

\usepackage{cleveref}

\makeatletter
\def\@tocline#1#2#3#4#5#6#7{\relax
  \ifnum #1>\c@tocdepth 
  \else
    \par \addpenalty\@secpenalty\addvspace{#2}%
    \begingroup \hyphenpenalty\@M
    \@ifempty{#4}{%
      \@tempdima\csname r@tocindent\number#1\endcsname\relax
    }{%
      \@tempdima#4\relax
    }%
    \parindent\z@ \leftskip#3\relax \advance\leftskip\@tempdima\relax
    \rightskip\@pnumwidth plus4em \parfillskip-\@pnumwidth
    #5\leavevmode\hskip-\@tempdima
      \ifcase #1
       \or\or \hskip 1em \or \hskip 2em \else \hskip 3em \fi%
      #6\nobreak\relax
    \hfill\hbox to\@pnumwidth{\@tocpagenum{#7}}\par
    \nobreak
    \endgroup
  \fi}
\makeatother

\renewcommand{\mod}{\ \textrm{mod}\ }

\newcommand{\Z}{\mathbb{Z}}

\newcommand{\m}{\mathfrak{m}}

\DeclareMathOperator{\Res}{Res}

\DeclareMathOperator{\Spec}{Spec}

\DeclareMathOperator{\Coker}{Coker}
\DeclareMathOperator{\Hom}{Hom}

\DeclareMathOperator{\Ker}{Ker}
\DeclareMathOperator{\reg}{reg}

\DeclareMathOperator{\sht}{ht}

\renewcommand{\Im}{\mathrm{Im}}

\theoremstyle{plain}
\newtheorem{theorem}{Theorem}[section]
\newtheorem{thm}[theorem]{Theorem}

\newtheorem{lem}[theorem]{Lemma}
\newtheorem{corollary}[theorem]{Corollary}

\newtheorem{question}[theorem]{Question}

\newtheorem*{claim*}{Claim}

\newtheorem{theoremA}{Theorem}

\theoremstyle{definition}
\newtheorem{definition}[theorem]{Definition}

\newtheorem*{setup*}{Setup}
\newtheorem{convention}[theorem]{Convention}

\theoremstyle{remark}
\newtheorem{remark}[theorem]{Remark}
\newtheorem*{ackn}{Acknowledgements}

\theoremstyle{plain}

\numberwithin{equation}{section}



\Crefname{theorem}{Theorem}{Theorems}
\Crefname{proposition}{Proposition}{Propositions}
\Crefname{lemma}{Lemma}{Lemmas}
\Crefname{corollary}{Corollary}{Corollaries}
\Crefname{conjecture}{Conjecture}{Conjectures}
\Crefname{claim}{Claim}{Claims}
\Crefname{notation}{Notation}{Notations}
\Crefname{remark}{Remark}{Remarks}
\Crefname{example}{Example}{Examples}
\Crefname{definition}{Definition}{Definitions}
\Crefname{theoremA}{Theorem}{Theorems}
\Crefname{question}{Question}{Questions}

\title[Quasi-$F^\infty$-split height versus quasi-$F$-regular height]{Quasi-$F^\infty$-split height versus quasi-$F$-regular height for rational double points and graded rings}

\author{Teppei Takamatsu}
\address{Department of Mathematics, Faculty of Science,
Saitama University,
255 Shimo-Okubo, Sakura-ku,
Saitama-shi, Saitama 338-8570,
Japan}
\email{teppeitakamatsu.math@gmail.com}

\author{Shou Yoshikawa}
\address{Institute of Science Tokyo, Tokyo 152-8551, Japan}
\email{yoshikawa.s.9fe9@m.isct.ac.jp}

\begin{document}

\begin{abstract}
In this paper, we study a phenomenon concerning quasi-$F$-singularities: under suitable hypotheses,
the finiteness of the quasi-$F^\infty$-split height ($\sht^{\infty}$) implies quasi-$F$-regularity, and moreover,
 $\sht^{\infty}$ coincides with the quasi-$F$-regular height ($\sht^{\reg}$).
We establish this coincidence for two important classes of isolated Gorenstein singularities.
First, we explicitly compute $\sht^{\infty}$ and $\sht^{\reg}$ for all rational double points, showing that every non-$F$-pure rational double point satisfies $\sht^\infty = \sht^{\reg}$.
Second, for localizations of graded non-$F$-pure normal Gorenstein rings with $F$-rational punctured spectrum, we again obtain the equality $\sht^\infty = \sht^{\reg}$.
\end{abstract}

\maketitle

\section{Introduction}

The notion of quasi-$F^e$-split height was introduced in \cite{Yobuko19} and \cite{TWY24}.
In \cite{TWY24}, it was shown that for a non-ordinary Calabi–Yau variety, the quasi-$F^e$-split height is strictly increasing in $e$, and in particular its limit, the quasi-$F^\infty$-split height, is infinite.
Motivated by this phenomenon, it is natural to ask when the quasi-$F^\infty$-split height of a singularity is finite, and what this finiteness should imply.

To formulate this question precisely, we recall the definition of the quasi-$F^e$-split height and the quasi-$F$-regular height.

\begin{definition}[cf.~\cite{TWY24}*{Lemma~3.10}]
Let $(R,\m)$ be a Gorenstein $F$-finite local ring of dimension $d$.
\begin{itemize}
\item For integers $n,e \geq 1$ and an element $c \in R$, we define homomorphisms $\Phi^e_{R,n}$ and $\Phi^{e,c}_{R,n}$ and $W_n(R)$-modules $Q^e_{R,n}$ and $Q^{e,c}_{R,n}$ by the following diagram in which every square is a pushout:
\begin{equation*}
\begin{tikzcd}
W_n(R) \arrow[r,"F^e"] \arrow[d,"\Res"'] &
F_*^e W_n(R) \arrow[d] \arrow[r,"{\cdot [c]}"] &
F_*^e W_n(R) \arrow[d] \\
R \arrow[r,"\Phi^{e}_{R,n}"] \arrow[rr,bend right, "\Phi^{e,c}_{R,n}"'] &
Q^e_{R,n} \arrow[r] & Q^{e,c}_{R,n}.
\end{tikzcd}
\end{equation*}
\item For integers $n,e \geq 1$, we say that $R$ is $n$-quasi-$F^e$-split if the homomorphism
\[
\Phi_{R,n}^e \colon H^d_\m(R) \to H^d_{W_n(\m)}(Q^e_{R,n})
\]
is injective.
Furthermore, we define the quasi-$F^e$-split height $\sht^e(R)$ of $R$ by
\[
\sht^e(R):=\inf\{n \geq 1 \mid \text{$R$ is $n$-quasi-$F^e$-split}\}
\]
if $R$ is $n'$-quasi-$F^e$-split for some $n'$, and $\sht^e(R):=\infty$ otherwise.
We simply denote $\sht^1$ by $\sht$.
Moreover, we define the quasi-$F^\infty$-split height by
\[
\sht^{\infty}(R) := \sup\{ \sht^e(R) \mid e \in \Z_{\geq 1} \} \in \Z_{>0} \sqcup \{\infty\}.
\]
\item For an integer $n \geq 1$, we say that $R$ is $n$-quasi-$F$-regular if for every $c \in R^\circ$ there exists an integer $e \geq 1$ such that the homomorphism
\[
\Phi^{e,c}_{R,n} \colon H^d_\m(R) \to H^d_{W_n(\m)}(Q^{e,c}_{R,n})
\]
is injective.
Furthermore, we define the quasi-$F$-regular height $\sht^{\reg}(R)$ of $R$ by
\[
\sht^{\reg}(R):=\inf\{n \geq 1 \mid \text{$R$ is $n$-quasi-$F$-regular}\}
\]
if $R$ is $n'$-quasi-$F$-regular for some $n'$, and $\sht^{\reg}(R):=\infty$ otherwise.
\end{itemize}
We note that we have
\[
1 \leq \sht(R) \leq \sht^2(R) \leq \cdots \leq \sht^{\infty}(R) \leq \sht^{\reg}(R) \leq \infty. 
\]
\end{definition}

The following natural question asks whether finiteness of the quasi-$F^{\infty}$-split height forces quasi-$F$-regularity.

\begin{question}\label{question}
Let $(R,\m)$ be a Gorenstein local ring of characteristic $p>0$, and assume that $R$ has an isolated singularity and is not $F$-pure.
If the quasi-$F^\infty$-split height of $R$ is finite, is $R$ quasi-$F$-regular?
Furthermore, do we have
\[
\sht^\infty(R)=\sht^{\reg}(R)?
\]
\end{question}

If we drop the assumption of non-$F$-purity, then the answer to this question is negative.
Indeed, if $E$ is an ordinary elliptic curve, then its affine cone $R$ is $F$-split but not quasi-$F$-regular (cf.\cite{KTTWYY3}*{Theorem~A}).
Thus $\sht^{\reg}(R)=\infty$ while $\sht^\infty(R)=1$.
On the other hand, for a supersingular elliptic curve $E$, it is known by \cite{TWY24}*{Theorem~7.1} that the affine cone satisfies that the quasi-$F^\infty$-height is infinite.


The goal of this paper is to study \Cref{question} for two important classes of singularities.
Our first main result provides a complete and explicit answer for rational double points (RDPs).
It is known by \cite{HaraRDP} that every taut RDP is strongly $F$-regular, thus $\sht^{\reg}=1$.
Furthermore, \cite{KTTWYY3}*{Theorem~C} shows that every RDP is quasi-$F$-regular in all characteristics.
Consequently, their quasi-$F$-regular height is always finite, and the only nontrivial part of \Cref{question} for RDPs is whether the equality $\sht^{\infty}=\sht^{\reg}$ holds.

The following theorem computes all quasi-$F^e$-split heights, the quasi-$F^{\infty}$-split height, and the quasi-$F$-regular height for every RDP in positive characteristic. 
In the computation, we use the Fedder-type criterion for quasi-$F^e$-splitting and quasi-$F$-regularity established in \cite{Yoshikawa25-fedder}.
Note that quasi-$F$-split heights of non-taut RDPs were computed in \cite{kty2}*{Table~1}.

\begin{theoremA}[\Cref{thm:qFeht}]\label{intro:thm:qFeht}
We completely determine the quasi-$F^e$-split heights and the quasi-$F$-regular heights
for non-taut RDPs as follows.
\begin{enumerate}
\item
For types other than type $D$ in characteristic $2$, the heights are given in
Table~\ref{table:RDPs}.
Moreover, for every non-$F$-pure RDP, we have
\[
\sht^e = \sht^{\infty} \qquad (e \geq 2).
\]

\item
For types $D_{2n}^0$ and $D_{2n+1}^0$, we have
\[
\sht^e = \sht^{\infty} = \sht^{\reg} = \lceil \log_2 n \rceil +1 \qquad (e \geq 1).
\]

\item
For types $D^r_{2n}$ and $D^r_{2n+1}$ \textup{($r=1,\ldots,n-1$)} in characteristic $p=2$,
the following statements hold.
\begin{itemize}
\item
We have
\[
\sht = \lceil \log_2 (n-r) \rceil +1.
\]

\item
Suppose that $n-r=1$.
Then
\[
\sht^e = \sht^{\infty} = \sht^{\mathrm{reg}} = 1 \qquad (e\geq 1).
\]

\item
Suppose that $n-r>1$ is a power of $2$.
Then
\[
\sht^e = \sht^{\infty} = \sht^{\mathrm{reg}}
= \log_2 (n-r) + 2 \qquad (e \geq 2).
\]

\item
Suppose that $n-r>1$ is not a power of $2$.
Define
\[
e_0 :=
\min
\Bigl\{
e \in \Z_{\geq 2}
\ \Big|\ 
2^{\lfloor\log_2(n-r)\rfloor + e}
- (2^{e}-1)(n-r)
+ (2^{e-1}-1)
<0
\Bigr\}.
\]
Then we have
\begin{equation*}
\sht^e =
\begin{cases}
\lceil \log_2 (n-r) \rceil + 1
& \textup{if } e < e_0, \\[0.3em]
\min \{ \lceil \log_2 (n-r) \rceil + 2,\ \lceil \log_2 n \rceil +1 \}
& \textup{if } e \geq e_0.
\end{cases}
\end{equation*}
Moreover,
\[
\sht^{\infty} = \sht^{\mathrm{reg}}
= \min \{ \lceil \log_2 (n-r) \rceil + 2,\ \lceil \log_2 n \rceil +1 \}.
\]
\end{itemize}
\end{enumerate}
\end{theoremA}


Our second main result gives a positive answer to \Cref{question} for localizations of certain graded rings.

\begin{theoremA}[\Cref{graded-case}, cf.~\cite{TWY24}*{Corollary~4.19}]\label{intro:graded}
Let $S$ be an $F$-finite Noetherian normal $\Z_{\ge 0}$-graded ring of characteristic $p>0$, and set $d := \dim S$.
Assume $S_0 = k$ is a field, put $\m := S_+$ and $R := S_\m$, and assume that $\Spec(R)\setminus\{\m\}$ is $F$-rational and that $R$ is Gorenstein.
If $R$ is not $F$-pure, then
\[
\sht^\infty(R)=\sht^{\reg}(R).
\]
\end{theoremA}


\begin{ackn}
The authors wish to express their gratitude to Christian Liedtke, Gebhard Martin,  Yuya Matsumoto, and Shunsuke Takagi for valuable discussions.
The first author was supported by JSPS KAKENHI Grant Number JP25K17228.
The second author was supported by JSPS KAKENHI Grant number JP24K16889.
\end{ackn}

\section{Quasi-$F^\infty$-split height and quasi-$F$-regular heights for RPDs}

In this section, we compute quasi-$F^\infty$-split heights and quasi-$F$-regular heights for RDPs.
As a consequence, we obtain $\sht^{\infty}(R)=\sht^{\reg}(R)$ for non-$F$-pure RDP $R$.

\subsection{Criteria for quasi-$F^\infty$-splitting and quasi-$F$-regularity}
In this subsection, we summarize criteria for quasi-$F$-splitting and quasi-$F$-regularity.

\begin{convention}
In this paper, $k$ is a perfect field of characteristic $p>0$, and
$R=k[[x_1,\ldots,x_N]]$ is a formal power series ring over $k$.
Let $\m=(x_1,\ldots,x_N)$ be the maximal ideal of $R$ generated by the variables.
Let $A=W(k)[[x_1,\ldots,x_N]]$ be a formal power series ring over $W(k)$, and
let $\pi \colon A \to R$ be the natural map.
By abuse of notation, the ideal $(x_1,\ldots,x_N)$ of $A$ is also denoted by $\m$.
We define a lift of the Frobenius morphism $\phi$ by
\[
\phi \colon A \to A\ \qquad x_i \mapsto x_i^p.
\]
Then $\phi_*A$ is a free $A$-module with basis
\[
\{\phi_*x_1^{i_1}\cdots x_N^{i_N} \mid 0 \leq i_j \leq p-1 \}.
\]
The dual basis element corresponding to $\phi_*(x_1\cdots x_N)^{p-1}$
with respect to the above basis is denoted by $u$.
In the same way, we define the corresponding map on $R$;
by abuse of notation, it is also denoted by $u$.
\end{convention}

We first recall criteria for quasi-$F^e$-splitting and quasi-$F$-regularity
established in \cite{Yoshikawa25-fedder}.

\begin{thm}[{\cite{Yoshikawa25-fedder}*{Theorem~A,C}}]\label{thm:criterion}
Let $n\ge1$ be an integer, and let $f \in A/p^n$ be a non-zero divisor.
\begin{enumerate}
    \item
    Let $e \ge 1$ be an integer.
    The ring $R/fR$ is $n$-quasi-$F^e$-split if and only if there exists $g \in A$
    satisfying the following conditions:
    \begin{itemize}
        \item[(D1)]
        $u^{e+r-1}(\phi^{e+r-1}_* g) \in (p^r)$ for $1 \le r \le n-1$;
        \item[(D2)]
        $g$ admits a decomposition
        \[
            g = g_0 + p g_1 + \cdots + p^{n-1} g_{n-1}
        \]
        such that
        $u^r(\phi^r_* g_r) \in (f^{p^{e+n-r-1} - 1})$
        for $0 \le r \le n-1$;
        \item[(D3)]
        $u^{e+n-2}(\phi_*^{e+n-2} g) \notin (\mathfrak{m}^{[p]}, p^n)$.
    \end{itemize}

    \item
    Let $\tau(R/fR)$ be the test ideal of $R/fR$, and let
    $t \in \tau(R,fR)$ and $c \in A$ be elements such that the image of $c$ by $A \to R/fR$ is contained in 
    \[
         (t^4) \cap (R/fR)^\circ.
    \]
    Then $R/fR$ is $n$-quasi-$F$-regular if and only if there exist
    $g \in A$ and an integer $e \geq 1$ such that $g$ satisfies condition
    \textup{(D2)}, and $g c^{p^n - 1}$ satisfies conditions
    \textup{(D1)} and \textup{(D3)} in (1).
\end{enumerate}
\end{thm}

\begin{corollary}[{\cite[Theorem B]{Yoshikawa25-fedder}}]\label{cor:non-qF^es condition}
Let $n, e \geq 1$ be integers, and let $f \in A/p^n$ be a non-zero divisor.
We define an $R$-module homomorphism $\theta$ by
\[
\theta \colon \Ker(u) \to R \ \qquad F_*a \mapsto u(F_*(\Delta_1(f^{p-1})a)).
\]
\begin{enumerate}
    \item
    If there exists $g \in f^{p^{e+n-1}-1} A/p^n$ satisfying conditions
    \textup{(D1)} and \textup{(D3)}, then $R/fR$ is $n$-quasi-$F^e$-split.
    \item
    Let $\tau(R/fR)$ be the test ideal of $R/fR$, and let
    $t \in \tau(R/fR)$ and $c \in A$ be elements such that the image of $c$ by $A \to R/fR$ is contained in 
    \[
         (t^4) \cap (R/fR)^\circ.
    \]
    If there exists $g \in f^{p^{e+n-1}-1} A/p^n$ such that
    $g c^{p^n-1}$ satisfies conditions \textup{(D1)} and \textup{(D3)},
    then $R/fR$ is $n$-quasi-$F$-regular.
    \item
    Define a sequence of ideals $\{I^e_n\}$ of $R$ inductively as follows.
    Set
    \[
    I^e_1 := f^{p-1} u^{e-1}(F_*^{e-1}f^{p^{e-1} - 1} R),
    \]
    and for each $n \geq 1$, define
    \[
    I^e_{n+1} := \theta\left(F_*( I^e_n \cap \Ker(u) )\right) + f^{p-1}R.
    \]
    If $R/fR$ is $n$-quasi-$F^e$-split, then
    $I^e_n \not\subseteq \mathfrak{m}^{[p]}$.

    \item
    Define another sequence of ideals $\{I'_n\}$ of $R$ as follows.
    Set
    \[
    I'_1 := f^{p-1} \mathfrak{m},
    \]
    and for each $n \geq 1$, define
    \[
    I'_{n+1} := \theta\left(F_*( I'_n \cap \Ker(u) )\right) + f^{p-1}R.
    \]
    If $R/fR$ is not $F$-pure and $I'_n \subseteq \mathfrak{m}^{[p]}$,
    then $R$ is not $n$-quasi-$F^2$-split.
\end{enumerate}
\end{corollary}

\begin{remark}
Let $n, e \geq 1$ be integers, and let $f \in A/p^n$ be a non-zero divisor.
If $g \in f^{p^{e+n-1}-1}A/p^n$, then $g$ satisfies condition \textup{(D2)}
by \cite{Yoshikawa25-fedder}*{Claim~5.5}.
Thus, \Cref{cor:non-qF^es condition}~(1) and (2) follow from
\Cref{thm:criterion}.
\end{remark}

\begin{lem}\label{lem:reduction lemma}
Let $a,f \in A$ with $f \notin pA$, and let $h \in \Z_{\geq 1}$.
For each integer $m \geq 0$, we write $f_m := f^{p^m-1}$.
Then the following conditions are equivalent:
\begin{itemize}
    \item[(i)]
    $u^{e+r-1}(f_{e+h-1}a) \in (p^r)$ for $1 \leq r \leq h-1$;
    \item[(ii)]
    $u^{e+r-1}(f_{e+r-1}a) \in (p^r)$ for $1 \leq r \leq h-1$; and
    \item[(iii)]
    $u^{e+r-1-s}(f_{e+r-1}a) \in (p^{r-s})$
    for $0 \leq s \leq r$ and $1 \leq r \leq h-1$.
\end{itemize}
\end{lem}

\begin{proof}
We use the notation $\Delta_m$ introduced in
\cite{Yoshikawa25}*{Theorem~3.6}.
By \cite{Yoshikawa25}*{Equation~(3.1)}, we have
\begin{equation}\label{eq:Delta}
    a^{p^m}
    =
    \sum_{l=0}^{m} p^m \phi^{m-l}(\Delta_l(a)).
\end{equation}
We prove the lemma by induction on $h$.
Thus, we may assume that conditions \textup{(i)}--\textup{(iii)} hold for
$r \leq h-2$.
It remains to verify them for $r=h-1$.

First, we show that condition \textup{(iii)} holds for $s \neq 0$.
Indeed, by \eqref{eq:Delta} and 
\[
f_{e+h-2}=f_s^{p^{e+h-2-s}}f_{e+h-2-s},
\]
 we have
\begin{align*}
    u^{e+h-2-s}\bigl(\phi_*^{e+h-2-s}(f_{e+h-2}a)\bigr)
    &=
    \sum_{t=0}^{e+h-2-s}
    p^t
    u^t\Bigl(
        \phi_*^{t}\bigl(
            \Delta_t(f_s)\,
            u^{e+h-2-s-t}\bigl(
                \phi_*^{e+h-2-s-t}(f_{e+h-2-s}a)
            \bigr)
        \bigr)
    \Bigr).
\end{align*}

By setting $r = h - 1 - s \leq h-2$ and $s=t$, we apply the induction hypothesis (iii) to obtain the following:
\[
p^t
u^{e+h-2-s-t}\bigl(
    \phi_*^{e+h-2-s-t}(f_{e+h-2-s}a)
\bigr)
\in (p^{h-1-s}).
\]

Next, we show that condition \textup{(i)} is equivalent to condition
\textup{(ii)}.
Indeed, by \eqref{eq:Delta}, we have
\begin{align*}
    u^{e+h-2}\bigl(\phi_*^{e+h-2}(f_{e+h-1}a)\bigr)
    &=
    \sum_{s=0}^{e+h-2}
    p^s
    u^s\Bigl(
        \phi_*^{s}\bigl(
            \Delta_s(f_1)\,
            u^{e+h-2-s}\bigl(
                \phi_*^{e+h-2-s}(f_{e+h-2}a)
            \bigr)
        \bigr)
    \Bigr) \\
    &\equiv
    f_1\,
    u^{e+h-2}\bigl(\phi_*^{e+h-2}(f_{e+h-2}a)\bigr)
    \mod p^{h-1}.
\end{align*}
Moreover, condition \textup{(iii)} for $s=0$ coincides with condition
\textup{(ii)}.
This completes the proof.
\end{proof}

\subsection{Quasi-$F^\infty$-split and quasi-$F$-regular heights for RDPs}

In this subsection, we compute the quasi-$F^e$-split heights and the quasi-$F$-regular heights for RDPs.
It is known that taut RDPs are strongly $F$-regular, thus we have $\sht^e=\sht^{\reg}=1$ for every $e \geq 1$ by \cite{HaraRDP}.

\begin{theorem}[\Cref{intro:thm:qFeht}]
\label{thm:qFeht}
We completely determine the quasi-$F^e$-split heights and the quasi-$F$-regular heights for non-taut RDPs as follows.
\begin{enumerate}
\item
For types other than type $D$ in characteristic $2$, the heights are given in
Table~\ref{table:RDPs}.
Moreover, for every non-$F$-pure RDP, we have
\[
\sht^e = \sht^{\infty} \qquad (e \geq 2).
\]

\item
For types $D_{2n}^0$ and $D_{2n+1}^0$, we have
\[
\sht^e = \sht^{\infty} = \sht^{\reg} = \lceil \log_2 n \rceil +1 \qquad (e \geq 1).
\]

\item
For types $D^r_{2n}$ and $D^r_{2n+1}$ \textup{($r=1,\ldots,n-1$)} in characteristic $p=2$,
the following statements hold.
\begin{itemize}
\item
We have
\[
\sht = \lceil \log_2 (n-r) \rceil +1.
\]

\item
Suppose that $n-r=1$.
Then
\[
\sht^e = \sht^{\infty} = \sht^{\mathrm{reg}} = 1 \qquad (e=1).
\]

\item
Suppose that $n-r>1$ is a power of $2$.
Then
\[
\sht^e = \sht^{\infty} = \sht^{\mathrm{reg}}
= \log_2 (n-r) + 2 \qquad (e \geq 2).
\]

\item
Suppose that $n-r>1$ is not a power of $2$.
Define
\[
e_0 :=
\min
\Bigl\{
e \in \Z_{\geq 2}
\ \Big|\ 
2^{\lfloor\log_2(n-r)\rfloor + e}
- (2^{e}-1)(n-r)
+ (2^{e-1}-1)
<0
\Bigr\}.
\]
Then we have
\begin{equation}
\label{eqn:Dnhte}
\sht^e =
\begin{cases}
\lceil \log_2 (n-r) \rceil + 1
& \textup{if } e < e_0, \\[0.3em]
\min \{ \lceil \log_2 (n-r) \rceil + 2,\ \lceil \log_2 n \rceil +1 \}
& \textup{if } e \geq e_0.
\end{cases}
\end{equation}
Moreover,
\[
\sht^{\infty} = \sht^{\mathrm{reg}}
= \min \{ \lceil \log_2 (n-r) \rceil + 2,\ \lceil \log_2 n \rceil +1 \}.
\]
\end{itemize}
\end{enumerate}
\end{theorem}

\begin{remark}
\label{rmk:e0}
The set appearing in the definition of $e_0$ is non-empty.
Indeed, we have
\[
\lim_{e \to \infty}
\left(
2^{\lfloor\log_2(n-r)\rfloor}
- \frac{2^{e}-1}{2^e}(n-r)
+ \frac{2^{e-1}-1}{2^e}
\right)
=
2^{\lfloor\log_2(n-r)\rfloor} - (n-r) + \frac{1}{2} < 0,
\]
where the last inequality follows from the assumption that $n-r$ is not a power of $2$.
Moreover, a direct computation shows that
\[
2^{\lfloor\log_2(n-r)\rfloor + e}
- (2^{e}-1)(n-r)
+ (2^{e-1}-1)
<0
\]
for any $e \geq e_0$.
\end{remark}

\begin{proof}[Proof of \Cref{thm:qFeht}]
Set $R:=k[[x,y,z]]$ and $\m:=(x,y,z)$.

First, we prove (1).
Let $f \in R$ be one of the elements listed in Table~\ref{table:RDPs}.
The quasi-$F$-split heights are given in \cite{kty}*{Table~1}.

If the pair $(p,\mathrm{Dyn}(R/f))$ is one of
\[
(2,E^1_7),\ (2,E^2_7),\ (2,E^2_8),\ (2,E^3_8),\ (3,E^1_8),
\]
then $R/f$ is not $\sht(R)$-quasi-$F^2$-split by \Cref{cor:non-qF^es condition}~(4).
In particular,
\[
\sht^2(R/f) \geq \sht(R/f)+1.
\]
In this case, $R/f$ is $(\sht(R)+1)$-quasi-$F$-regular by
\Cref{cor:non-qF^es condition}~(2), and hence
\[
\sht^{\reg}(R/f)=\sht^2(R/f).
\]
Indeed, take $e$, $a$, and $c$ as in \Cref{table:RDPs-qfr}, and set
$n:=\sht(R)+1$ and $g:=af^{p^{e+n-1}}$.
Then $gc^{p^n-1}$ satisfies conditions (D1) and (D3) in \Cref{thm:criterion}.

On the other hand, if the pair $(p,\mathrm{Dyn}(R/f))$ is one of
\[
(2,E^0_6),\ (2,E^0_7),\ (2,E^0_8),\ (2,E^1_8),\ (3,E^0_6),\ (3,E^0_7),\ (3,E^0_8),\ (5,E^0_8),
\]
then $R/f$ is $\sht(R)$-quasi-$F$-regular.
Indeed, take $e$, $a$, and $c$ as in \Cref{table:RDPs-qfr}, and set
$n:=\sht(R)$ and $g:=af^{p^{e+n-1}}$.
Then $gc^{p^n-1}$ satisfies conditions (D1) and (D3) in \Cref{thm:criterion}.

If the pair $(p,\mathrm{Dyn}(R/f))$ is one of
\[
(2,E^1_6),\ (3,E^1_6),\ (3,E^1_7),\ (3,E^2_8),\ (5,E^1_8),
\]
then $R/f$ is $F$-pure, and hence $\sht^e(R/f)=\sht^{\infty}(R/f)=1$.
Furthermore, $R/f$ is $2$-quasi-$F$-regular by \Cref{cor:non-qF^es condition}~(2), and in particular,
\[
\sht^{\reg}(R/f)=2.
\]
Indeed, take $e$, $a$, and $c$ as in \Cref{table:RDPs-qfr}, and set
$n:=2$ and $g:=af^{p^{e+n-1}}$.
Then $gc^{p^n-1}$ satisfies conditions (D1) and (D3) in \Cref{thm:criterion}.

Finally, assume that the pair $(p,\mathrm{Dyn}(R/f))$ is one of
\[
(2,E^3_7),\ (2,E^4_8).
\]
Then $R/f$ is $3$-quasi-$F$-regular.
Indeed, take $e$, $a$, and $c$ as in \Cref{table:RDPs-qfr}, and set
$n:=3$ and $g:=af^{p^{e+n-1}}$.
Then $gc^{p^n-1}$ satisfies conditions (D1) and (D3) in \Cref{thm:criterion}.

We show that type $(2,E^3_7)$ is not $2$-quasi-$F$-regular.
Suppose that $(R/f)_\m$ is $2$-quasi-$F$-regular, where $f=z^2+x^3+xy^3+xyz$.
Take $c:=x^4 \in A$.
By \Cref{thm:criterion}, there exist $g \in A$ and an integer $e \geq 1$ such that
$g$ satisfies condition (D2) and $gc^3$ satisfies conditions (D1) and (D3).
By condition (D2), we can write
\[
g=g_0+pg_1
\]
so that $g_0 \in (f^{2^{e+1}-1})$ and $u(\phi_*g_1) \in (f^{2^{e}-1})$.
Since
\begin{align*}
    f^{2^{e+1}-1}
    &= f^{2^{e}}f^{2^{e-1}} \cdots f^2 f \\
    &\equiv \Bigl((xyz)^{2^e}+2\bigl((x^3z^2)^{2^{e-1}}+(xy^3z^2)^{2^{e-1}}+(xyz^3)^{2^{e-1}}\bigr)\Bigr)f^{2^{e-1}} \cdots f^2 f
    \pmod{(\m^{[2^{e+1}]},4)} \\
    &\equiv (xyz)^{2^e}f^{2^{e-1}} \cdots f^2 f
    \pmod{(\m^{[2^{e+1}]},4)},
\end{align*}
where the last congruence follows from
\[
\bigl((x^3z^2)^{2^{e-1}}+(xy^3z^2)^{2^{e-1}}+(xyz^3)^{2^{e-1}}\bigr)f^{2^{e-1}} \in (\m^{[2^{e+1}]},2),
\]
we obtain
\begin{equation}\label{eq:e+1}
    f^{2^{e+1}-1} \equiv (xyz)^{2^{e+1}-1} \pmod{(\m^{[2^{e+1}]},4)}.
\end{equation}
Furthermore, we have
\begin{equation}\label{eq:e}
    f^{2^{e}-1} \equiv (xyz)^{2^{e}-1} \pmod{(\m^{[2^{e}]},2)}.
\end{equation}
Therefore,
\[
u^e(\phi^e_*(c^3g_0)) \in (\m^{[2]},4)
\]
by \eqref{eq:e+1}, and
\[
u^e(\phi^e_*(pc^3g_1))
=pu^{e-1}\bigl(\phi^{e-1}_*(x^6u(\phi_*(g_1)))\bigr)
\in (\m^{[2^{e}]},4)
\]
by \eqref{eq:e}.
The case $(2,E^4_8)$ is also not $2$-quasi-$F$-regular by the same argument.

Next, we prove (3).
We first treat the case of type $D_{2n}^r$.
Let
\[
f = x^2y+xy^n+xy^{n-r}z+z^2\in R.
\]
Let $J_e:=u^{e-1}(F_*^{e-1}f^{p^{e-1} - 1} R)$ be an ideal of $R$.
Then $I_e=f^{p-1}J_e$ for every $e \geq 1$.

The computation of $\sht(R/f)$ follows from \cite{kty}.
Moreover, the case $n-r=1$ follows from the classical Fedder criterion.
In what follows, we fix $n$, $r$, and an integer $e \geq 2$ such that $n-r>1$.

First, we compute $\sht^e(R/f)$.
We denote $\log_2 (n-r) +2$ (resp.\ the right-hand side of \eqref{eqn:Dnhte}) by $h_e$
if $n-r$ is a power of $2$ (resp.\ if $n-r$ is not a power of $2$).
We prove $\sht^e(R/f) = h_e$.

\noindent{\bf Proof of $\sht^e(R/f) \geq h_e$.}
Note that $h_e$ is equal to $\lceil\log_2 (n-r)\rceil +1$ or $\lceil\log_2 (n-r)\rceil +2$.
Since $\sht^e \geq \sht^1 = \lceil \log_2 (n-r) \rceil+1$,
we may assume that $h_e = \lceil \log_2 (n-r)\rceil+2$, that is,
either $n-r$ is a power of $2$ or $e \geq e_0$ holds.
Suppose, to the contrary, that $R/f$ is quasi-$F^e$-split at $h_e-1$.
By \Cref{cor:non-qF^es condition}, we have
\[
\theta^{h_e-2} (F^{h_e-2}_*I^e_1) \nsubseteq \m^{[p]}.
\]
Therefore, it suffices to show
\[
f\Delta_1 (f)^{1 + 2 + \cdots 2^{h_e-3}} J_e\subset \m^{[2^{h_e-1}]}.
\]
Note that $h_e \geq 3$.
Each term of $f\Delta_1 (f)^{1 + 2 + \cdots 2^{h_e-3}}$ is obtained by multiplying $2^{h_e-1}-1$ terms of $f$, with repetitions allowed.
Hence each term can be written as
\[
(x^2y)^A(xy^n)^B(xy^{n-r}z)^C(z^2)^D,
\]
where $A,B,C,D$ are non-negative integers satisfying $A+B+C+D=2^{h_e-1}-1$.
If such a monomial is not contained in $\m^{[2^{h_e-1}]}$, then we have
\begin{equation}
\label{eqn:ABCDinequality}
2A+C \leq 2^{h_e-1}-1,\quad
A+nB+(n-r)C \leq 2^{h_e-1}-1,\quad
B+C+2D \leq 2^{h_e-1}-1.
\end{equation}
Since $h_e = \lceil \log_2(n-r) \rceil+2$, we have
\begin{equation}
\label{eqn:h_einequality}
2(n-r) \leq 2^{h_e-1} <4(n-r).
\end{equation}
By the second inequality of \eqref{eqn:ABCDinequality} and \eqref{eqn:h_einequality}, we have $B+C \leq 3$.

Suppose that $2 \leq B+C$.
Then, by the third inequality of \eqref{eqn:ABCDinequality}, we have
\[
D\leq 2^{h_e-2}-2.
\]
Moreover, by the second inequality of \eqref{eqn:ABCDinequality} and \eqref{eqn:h_einequality}, we obtain
\[
A \leq 2^{h_e-1}-1 - 2 (n-r) < 2^{h_e-2}-1.
\]
Therefore,
\[
A+B+C+D \leq (2^{h_e-2}-2) + 3 + (2^{h_e-2}-2) = 2^{h_e-1} -1.
\]
Since equality holds, we have $A = 2^{h_e-2}-2$, $B+C=3$, and $D = 2^{h_e-2}-2$.
By the second inequality of \eqref{eqn:ABCDinequality} and \eqref{eqn:h_einequality} again, we have
\[
2^{h_{e}-3} (B+C) < (n-r)(B+C) \leq 2^{h_{e}-2}+1,
\]
which implies $B+C < 3$, a contradiction.
Hence we must have $B+C=1$.

In this case, by \eqref{eqn:ABCDinequality}, we have
\[
A \leq 2^{h_e-2}-1,\quad D \leq 2^{h_e-2}-1.
\]
As before, equality must hold in both inequalities, and therefore we have two possibilities:
\[
(A,B,C,D)=(2^{h_e-2}-1,1,0,2^{h_e-2}-1),\quad
(2^{h_e-2}-1,0,1,2^{h_e-2}-1).
\]
In the first case, by the second inequality of \eqref{eqn:ABCDinequality}, we have
$n \leq 2^{h_e-2}= 2^{\lceil \log_2 (n-r)\rceil}$, which implies
$\lceil \log_2 (n-r) \rceil = \lceil \log_2 n \rceil$.
Since $h_e$ is defined by the right-hand side of \eqref{eqn:Dnhte} in this case, this does not occur.
Therefore, we must have
\[
(A,B,C,D) =(2^{h_e-2}-1,0,1,2^{h_e-2}-1),
\]
and the corresponding term is
\[
x^{2A+C}y^{A+nB+(n-r)C}z^{B+C+2D}
=
x^{2^{h_e-1}-1}y^{2^{h_e-2}-1+(n-r)} z^{2^{h_e-1}-1}.
\]
Thus, it suffices to show
\[
J_e \subset (x,y^{2^{h_e-2}-(n-r) +1},z)
= (x, y^{2^{\lceil \log_2 (n-r) \rceil}-(n-r) +1},z).
\]
To this end, we prove the following lemma.

\begin{lem}
\label{lem:computationJm}
For any $n$, $r$, and $m \geq 2$, we have
\begin{equation}
\label{eqn:Jm}
J_m = (x,y^{\alpha_m}, y^{\lfloor \frac{n}{2} \rfloor}, z),
\end{equation}
where $\alpha_m$ is defined by $\alpha_2 = \lfloor \frac{n-r}{2} \rfloor$ and
\[
\alpha_{m+1} = \lfloor \frac{n-r + \alpha_m}{2} \rfloor.
\]
\end{lem}

\begin{proof}
Note that $\alpha_m \leq  n-r$ for any $m$, and hence
\[
\alpha_{m+1} \geq \Bigl\lfloor \frac{2\alpha_m}{2} \Bigr\rfloor = \alpha_m.
\]
Recall that $J_1 = (1)$ and $J_{m+1} = u (F_*(fJ_m))$.
By a straightforward computation, we have
$J_2 = (x,y^{\lfloor \frac{n-r}{2}\rfloor}, z)$, so \eqref{eqn:Jm} holds for $m=2$.
Assume that \eqref{eqn:Jm} holds for $m$, and we prove it for $m+1$.

First, note that
\[
u(F_* (x f)) = (x, z), \quad u(F_* (zf)) = (x, y^{\lfloor \frac{n}{2} \rfloor}, z).
\]
Moreover, we obtain
\[
u(F_* (y^{\alpha_m} f)) + (x,z) = (x, y^{\lfloor \frac{n-r + \alpha_m }{2} \rfloor}, z).
\]
Therefore, \eqref{eqn:Jm} for $m+1$ holds when $\alpha_m \leq \lfloor \frac{n}{2} \rfloor$.

Suppose that $\lfloor \frac{n}{2} \rfloor < \alpha_m$.
Then there exists an integer $m_0 \geq 3$ such that
$\alpha_{m_0-1} \leq \lfloor \frac{n}{2} \rfloor$ and
$\lfloor \frac{n}{2}  \rfloor < \alpha_{m_0}$.
Since
\[
u(F_* (y^{\lfloor \frac{n}{2} \rfloor} f)) + (x,z)
=
(x, y^{\lfloor \frac{n-r + \lfloor \frac{n}{2} \rfloor }{2} \rfloor}, z)
\]
and
\[
\lfloor \frac{n-r + \lfloor \frac{n}{2} \rfloor }{2} \rfloor
\geq
\lfloor \frac{n-r + \alpha_{m_0-1} }{2} \rfloor
= \alpha_{m_0}
> \lfloor \frac{n}{2} \rfloor,
\]
we obtain
\[
J_{m+1} = (x, y^{\lfloor \frac{n}{2} \rfloor}, z),
\]
as desired.
\end{proof}

By Lemma~\ref{lem:computationJm}, to prove $\sht^e(R/f) \geq h_e$, it suffices to show
\begin{equation}
\label{eqn:firstineq}
\alpha_e \geq 2^{\lceil \log_2 (n-r) \rceil}- (n-r) +1
\end{equation}
and
\begin{equation}
\label{eqn:secondineq}
\lfloor \frac{n}{2} \rfloor \geq 2^{\lceil \log_2 (n-r) \rceil}- (n-r) +1.
\end{equation}
If $n-r > 1$ is a power of $2$, these inequalities are immediate.
Thus, we may assume that $n-r >1$ is not a power of $2$.

Since we assume $h_e = \lceil \log_2 (n-r) \rceil +2$, we have $e\geq e_0$ and
\begin{equation}
\label{eqn:log2nineq}
\lceil \log_2 n \rceil \geq \lceil \log_2 (n-r) \rceil +1.
\end{equation}
Since
\[
2^{\lceil \log_2 (n-r) \rceil -1} < n-r \leq 2^{\lceil \log_2 (n-r) \rceil}
\]
and $2^{\lceil \log_2 (n-r) \rceil} < n$ (which follows from \eqref{eqn:log2nineq}), we obtain
\[
2^{\lceil \log_2 (n-r) \rceil}- (n-r) +1
<
2^{\lceil \log_2 (n-r) \rceil} - 2^{\lceil \log_2 (n-r) \rceil -1} +1
< \frac{n}{2}+1.
\]
This implies \eqref{eqn:secondineq}.
Therefore, it remains to prove \eqref{eqn:firstineq}.

Since $e \geq e_0$, Remark~\ref{rmk:e0} implies
\[
2^{\lceil \log_2 (n-r) \rceil + e-1} - (2^e-1) (n-r) + (2^{e-1}-1) <0.
\]
Hence
\[
2^{\lceil \log_2 (n-r) \rceil} - (n-r) +1 <
\Bigl(\frac{2^{e-1}-1}{2^{e-1}}\Bigr) (n-r) + \frac{1}{2^{e-1}}.
\]

\begin{lem}
\label{lem:computationalpha_m}
For any $m\geq2$, write $n-r = 2^{m-1}M_m + N_m$, where $0 \leq N_m \leq 2^{m-1}-1$.
Then
\begin{equation}
\label{eqn:alpha_m}
\alpha_m = (2^{m-1}-1)M_m +
\begin{cases}
0  &\textup{if } 0 \leq  N_m \leq 1, \\
N_m-1 &\textup{otherwise}.
\end{cases}
\end{equation}
\end{lem}

\begin{proof}
We prove \eqref{eqn:alpha_m} by induction on $m$.
Since $\alpha_2 = \lfloor \frac{n-r}{2} \rfloor = M_2$, the case $m=2$ holds.
Assume that \eqref{eqn:alpha_m} holds for $m$.

If $0 \leq N_{m+1} \leq 2^{m-1} -1$ (resp.\ $2^{m-1} \leq N_{m+1} \leq 2^{m}-1$), then
$M_m = 2M_{m+1}$ (resp.\ $M_m = 2M_{m+1} +1$) and
$N_m = N_{m+1}$ (resp.\ $N_m = N_{m+1} - 2^{m-1}$).
Therefore, by \eqref{eqn:alpha_m} for $m$, we have
\[
\alpha_m =
\begin{cases}
(2^{m}-2)M_{m+1}  &\textup{if } 0 \leq N_{m+1} \leq 1, \\
(2^m-2)M_{m+1} + N_{m+1}-1  &\textup{if } 2 \leq  N_{m+1} \leq 2^{m-1}-1,\\
(2^{m-1}-1)(2M_{m+1}+1)  &\textup{if } 2^{m-1} \leq  N_{m+1} \leq  2^{m-1}+1, \\
(2^{m-1}-1)(2M_{m+1}+1)+ (N_{m+1}-2^{m-1}-1)
&\textup{if } 2^{m-1}+2 \leq  N_{m+1} \leq  2^{m}-1.
\end{cases}
\]
Combining this with $\alpha_{m+1} = \lfloor \frac{n-r + \alpha_m}{2} \rfloor$, we obtain \eqref{eqn:alpha_m} for $m+1$.
\end{proof}

Since $n-r = 2^{e-1} M_e + N_e$, we have
\[
2^{\lceil \log_2 (n-r) \rceil} - (n-r) +1 <
(2^{e-1}-1)M_e
+
\Bigl(\frac{2^{e-1}-1}{2^{e-1}}\Bigr) N_e + \frac{1}{2^{e-1}}.
\]
By Lemma~\ref{lem:computationalpha_m}, this implies \eqref{eqn:firstineq}.
This proves $\sht^e(R/f) \geq h_e$.

\noindent{\bf Proof of $\sht^e(R/f) \leq h_e$.}
In what follows, we also denote the polynomial
\[
x^2y+xy^n+xy^{n-r}z+z^2 \in W(k)[x,y,z]
\]
by $f$, and we write $f_m:=f^{p^m-1}$.

\noindent{\bf Case (1).}
Assume that one of the following holds:
\begin{itemize}
    \item $n-r\geq 2$ is a power of $2$;
    \item $n-r\geq 2$ is not a power of $2$ and $e < e_0$.
\end{itemize}
In this case, $h_e = \lfloor \log_2 (n-r) \rfloor +2$.
Set
\begin{equation}
\label{eqn:Case1g}
g:=  f_{e+h_e-1}
y^{2^{\lfloor \log_2(n-r) \rfloor +e} - (2^e-1)(n-r) + (2^{e-1} -1)}.
\end{equation}
Note that $g$ is well-defined by the assumption.
We show that $g$ satisfies conditions (D1) and (D3) in Theorem~\ref{thm:criterion}, and hence $R/f$ is $h_e$-quasi-$F^e$-split by \Cref{cor:non-qF^es condition}~(1).

By Lemma~\ref{lem:reduction lemma}, to prove (D1), it suffices to show that
\[
u^{e+s-1}\bigl(\phi^{e+s-1}_*(f_{e+s-1}y^{2^{\lfloor \log_2(n-r) \rfloor +e} - (2^e-1)(n-r) + (2^{e-1} -1)})\bigr) \in (p^s)
\]
for $1 \leq s \leq h_e-1$.
We compute the terms of
\[
f_{e+s-1}y^{2^{\lfloor \log_2(n-r) \rfloor +e} - (2^e-1)(n-r) + (2^{e-1} -1)}
\]
whose image under $u^{e+s-1}$ is non-zero.
Ignoring coefficients, such a term can be written as
\begin{equation}
\label{eqn:termcase1}
(x^2y)^A(xy^n)^B(xy^{n-r}z)^C(z^2)^D
y^{2^{\lfloor \log_2(n-r) \rfloor +e} - (2^e-1)(n-r) + (2^{e-1} -1)},
\end{equation}
where $A,B,C,D$ are non-negative integers with $A+B+C+D=2^{e+s-1}-1$ and
\[
2A+B+C \equiv A+nB+(n-r)C-(2^e-1)(n-r)+(2^{e-1}-1) \equiv C+2D \equiv -1 \mod 2^{e+s-1}.
\]
Modulo $2^{e+s-1}$, the congruence $A+B+C+D\equiv 2A+B+C$ implies $D \equiv A$.
The congruence $C+2D \equiv -1$ implies $D\equiv -1 -2A$, and
$2A+B+C \equiv -1$ implies $B \equiv 0$.
Moreover, since
\[
A+nB+(n-r)C-(2^e-1)(n-r)+(2^{e-1}-1)
\equiv (1-2(n-r))(A+2^{e-1})-1 \equiv-1 \mod 2^{e+s-1},
\]
we obtain
\[
A \equiv -2^{e-1} \mod 2^{e+s-1}.
\]
Hence $A = 2^{e+s-1} -2^{e-1}$, and also $C= 2^{e+s-1} -2^{e-1}$ since $A \equiv C \mod 2^{e+s-1}$.
Then
\[
A+C= 2^{e+s}-2^e >2^{e+s-1}-1 = A+B+C+D,
\]
a contradiction.
Therefore, there is no such term, and (D1) follows.

Next, to verify (D3), we compute $u^{e+h_e-1}(\phi^{e+h_e-1}_*g)$.
By the same argument as above, ignoring coefficients, any term of $g$ whose image under $u^{e+h_e-1}$ is non-zero can be written as \eqref{eqn:termcase1} with
$A+B+C+D=2^{e+h_e-1}-1$ and
\begin{gather*}
2A+B+C\equiv C+2D \equiv -1 \mod 2^{e+h_e-1}, \\
A+nB+(n-r)C+ 2^{e+h_e-2}-(2^e-1) (n-r)+(2^{e-1}-1) \equiv -1 \mod 2^{e+h_e-1}.
\end{gather*}
As before, we obtain $A \equiv D$, $D \equiv -1-2A$, and $B \equiv 0 \mod 2^{e+h_e-1}$.
Moreover,
\begin{align*}
& A+nB+(n-r)C + 2^{e+h_e-2}-(2^e-1)(n-r)+(2^{e-1}-1) \\
&\equiv (1-2(n-r))(A+2^{e-1}) +2^{e+h_e-2} -1 \\
&\equiv -1 \mod 2^{e+h_e-1}.
\end{align*}
and hence
\[
A \equiv -2^{e-1} \mod 2^{e+h_e-2}.
\]
Since $A=D$ and $A+B+C+D = 2^{e+h_e-1}- 1< 2^{e+h_e}-2^e$, we obtain
$A= 2^{e+h_e-2}-2^{e-1} =D$.
Therefore,
\[
(A,B,C,D) = (2^{e+h_e-2}-2^{e-1}, 0, 2^e-1,2^{e+h_e-2}-2^{e-1}).
\]
Furthermore, the coefficient of this term is
\[
\binom{2^{e+h_e-1}-1}{A\ C\ D}
=
\binom{2^{e+h_e-1}-1}{2^e-1}
\binom{2^{e}(2^{h_e-1}-1)}{2^{e-1}(2^{h_e-1}-1)}.
\]
By Kummer's theorem, the $2$-order of this coefficient is $h_e-1$.
Therefore,
\[
2^{-(h_e-1)}u^{e+h_e-1}(\phi^{e+h_e-1}_*g)
\notin (\m,2),
\]
as desired.

\noindent{\bf Case (2).}
Put $h' = \lceil \log_2 n \rceil+1$, and define
\begin{equation}
\label{eqn:Case2g}
g:= f_{e+h' -1} x^{2^{e-1}-1} y^{(2^{\lceil \log_2 n \rceil}-n)2^{e-1}+2^{e-1}-1}z.
\end{equation}
If $g$ satisfies (D1) and (D3) (for $h'$) in Theorem~\ref{thm:criterion}, then
$\sht \leq h'$ by \Cref{cor:non-qF^es condition}~(1).
In particular, we obtain $\sht(R/f) \leq h_e$ in the case where
$\lceil \log_2n \rceil = \lceil \log_2 (n-r) \rceil$.

To verify (D1), it suffices to show that the image of
\begin{equation}
\label{eqn:case2fs}
f_{e+s-1} x^{2^{e-1}-1} y^{(2^{\lceil \log_2 n \rceil}-n)2^{e-1}+2^{e-1}-1}z
\end{equation}
under $u^{e+s-1}$ is contained in $(p^s)$ for every $1\leq s \leq h'-1$.
Ignoring coefficients, the terms in \eqref{eqn:case2fs} whose image under $u^{e+s-1}$ is non-zero can be written as
\begin{equation}
\label{eqn:termcase2}
(x^2y)^A(xy^n)^B(xy^{n-r}z)^C(z^2)^D
x^{2^{e-1}-1} y^{(2^{\lceil \log_2 n \rceil}-n)2^{e-1}+2^{e-1}-1}z,
\end{equation}
where $A,B,C,D$ are non-negative integers with $A+B+C+D = 2^{e+s-1}-1$ and
\begin{eqnarray*}
&&2A+B+C+2^{e-1}-1 \\
&\equiv& A+nB+(n-r)C-n2^{e-1}+2^{e-1}-1 \\
&\equiv& C+2D+1 \equiv -1 \mod 2^{e+s-1}.
\end{eqnarray*}
As in Case~(1), we obtain
$D\equiv A+2^{e-1}-1$, $C\equiv  -2A-2^e$, and $B\equiv2^{e-1} \mod 2^{e+s-1}$.
Moreover, since
\[
A+nB+(n-r)C-n2^{e-1}+2^{e-1}-1
\equiv (1-2(n-r))(A+2^{e-1}) -1 \equiv -1,
\]
we obtain $A\equiv -2^{e-1} \mod 2^{e+s-1}$.
Hence $A= 2^{e+s-1}-2^{e-1}$ and $D=2^{e+s-1}-1$, which contradicts
\[
A+B+C+D = 2^{e+s-1}-1 < 2^{e+s}-2^{e-1}-1.
\]
Therefore, there is no such term, and (D1) holds.

Next, we verify (D3).
As in Case~(1), the terms in $g$ whose image under $u^{e+h'-1}$ is non-zero can be written as \eqref{eqn:termcase2} with $A+B+C+D=2^{e+h'-1}-1$ and
\begin{gather*}
2A+B+C+2^{e-1}-1 \equiv C+2D+1 \equiv -1 \mod 2^{e+h'-1},\\
A+nB+(n-r)C+2^{e+h'-2}-n2^{e-1}+2^{e-1}-1 \equiv -1 \mod 2^{e+h'-1}.
\end{gather*}
As before, we obtain $D\equiv A+2^{e-1}-1$, $C\equiv  -2A-2^e$, and $B\equiv2^{e-1} \mod 2^{e+h'-1}$.
Moreover, we have $A \equiv -2^{e-1} \mod 2^{e+h'-2}$.
Since $D=A+2^{e-1}-1$ and
\[
A+B+C+D=2^{e+h'-1}-1<2^{e+h'}-2^{e-1}-1,
\]
we obtain $A=2^{e+h'-2}-2^{e-1}$.
Therefore,
\[
(A,B,C,D) = (2^{e+h'-2}-2^{e-1}, 2^{e-1},0, 2^{e+h'-2}-1).
\]
The coefficient of this term is
\[
\binom{2^{e+h'-1}-1}{A\ C\ D}
=
\binom{2^{e+h'-1}-1}{2^{e+h'-2}-1}
\binom{2^{e+h'-2}}{2^{e-1}}.
\]
By Kummer's theorem, the $2$-order of this coefficient is $h'-1$.
This proves (D3).

\noindent{\bf Case (3).}
Put $h'':= \lceil \log_2(n-r) \rceil+2$, and define
\begin{equation}
\label{eqn:Case3g}
g:= f_{e+h''-1}
y^{2^{\lfloor \log_2(n-r) \rfloor +e+1} - (2^e-1)(n-r) + (2^{e-1} -1)}.
\end{equation}
In this case, conditions (D1) and (D3) (for $h''$) in Theorem~\ref{thm:criterion}
can be verified in the same way as in Case~(1).
Therefore, $\sht(R/f) \leq h''$ by \Cref{cor:non-qF^es condition}~(1), and hence
$\sht(R/f) \leq h_e$ in the case where $n-r\geq 2$ is not a power of $2$ and $e \geq e_0$.

By Cases~(1)--(3), we obtain $\sht^e(R/f) \leq h_e$ in all cases.
Hence $\sht^e(R/f)=h_e$.

\noindent{\bf Proof of $\sht^{\infty}(R/f) = \sht^{\mathrm{reg}}(R/f)$.}
Take $e \gg 0$ such that $\sht^{e}(R/f) = \sht^{\infty}(R/f)$.
It suffices to show $\sht^{\mathrm{reg}}(R/f) = \sht^{\infty}(R/f)$.
Let $\tau \subset R/f$ be the test ideal.
We have
\[
\partial_x f = y^n +y^{n-r}z,\quad
\partial_y f = x^2 + nxy^{n-1} + (n-r)xy^{n-r-1}z,\quad
\partial_z f = xy^{n-r}.
\]
Let
\[
u \colon F_*R \to R
\]
be the $R$-module homomorphism given by
\[
x^i y^j z^k \mapsto
\begin{cases}
x^{\frac{i-1}{2}}y^{\frac{j-1}{2}}z^{\frac{k-1}{2}} \quad & \textup{if $i,j,k$ are odd,} \\
0 &\textup{otherwise.}
\end{cases}
\]
Then $- \mapsto u (F_*(f-))$ defines a generator of
\[
\Hom_{(R/f)_{\m}} (F_* ((R/f)_{\m}), (R/f)_{\m}).
\]
Set $\alpha = 1$ (resp.\ $0$) if $n-r$ is even (resp.\ odd).
Then
\[
u(F_*(fxy^{n-r} y^{\alpha}z)) = y^{\lfloor \frac{n-r}{2} \rfloor}z \in \tau.
\]
Since $\partial_x f \in \tau$, we have $y^n \in \tau$.
Let $c:=y^{4n}$.

First, suppose that $n-r \geq 2$ is a power of $2$.
Then $g$ defined in \eqref{eqn:Case1g} satisfies
\[
\frac{g}{c^{2^{h_e}-1}} \in (f_{e+h_e-1}) \subset W(k)[x,y,z]
\]
for $e \gg 0$.
Note that $h_e = \sht^e(R/f)$.
Therefore, $\frac{g}{c^{2^{h_e}-1}}$ and $c$ satisfy the condition in \Cref{thm:criterion}~(2), and hence
\[
\sht^{\mathrm{reg}}(R/f) = \sht^{e}(R/f) = \sht^{\infty}(R/f).
\]
When $n-r$ is not a power of $2$, the same argument works by using $g$ defined in \eqref{eqn:Case3g}.
This completes the proof for type $D_{2n}^r$.

The proof for type $D_{2n+1}^r$ is similar.
Here we only indicate the modifications.
As before, we may assume $n-r \geq 2$, and we set
\[
f := z^2 + x^2y + y^nz +xy^{n-r}z.
\]
The proof of $\sht^e(R/f) \geq h_e$ is the same.
For the proof of $\sht^e(R/f) \leq h_e$, in Cases~(1) and (3), the same definitions of $g$
\eqref{eqn:Case1g} and \eqref{eqn:Case3g} work.
In Case~(2), instead of \eqref{eqn:Case2g}, we use
\[
g = f_{e+h'-1} xy^{(2^{\lceil \log_{2}n \rceil}-n)2^{e-1}}z^{2^{e-1}-1}.
\]
Then, by a similar computation, we obtain $\sht^e(R/f) \leq h'$.
For the computation of the test ideal, we have
\[
u(F_*(f (\partial_{x}f)xy^{1-\alpha})) = xy^{\lfloor \frac{n-r+1}{2} \rfloor} \in \tau.
\]
Since $\partial_z f = y^n + x y^{n-r} \in \tau$, we obtain $y^n \in \tau$.
Hence $\sht^{\infty}(R/f) = \sht^{\mathrm{reg}}(R/f)$ follows in the same way.
This completes the proof of (3).

Finally, we prove (2).
Since $\sht = \lceil \log_{2} n \rceil +1$ by \cite[Table~1]{kty}, it suffices to show
$\sht^{\reg} \leq \lceil \log_{2} n \rceil +1$.

First, consider type $D_{2n}^0$ and set
\[
f:=z^2 +x^2y +xy^n.
\]
Then, for any $e\geq 2$, the element
\begin{equation}
\label{eqn:GD0}
g:= z y^{2^{\lceil \log_2 n \rceil+e-1} -n} f^{2^{\lceil \log_2n \rceil +e} -1}
\end{equation}
satisfies conditions (D1) and (D3) in Theorem~\ref{thm:criterion} by the same argument as in (3).
Therefore, $\sht^{e}(R/f) \leq \lceil \log_{2} n \rceil +1$.
Moreover, since $\partial_x f = y^n \in \tau$, and $c := y^{4n}$ satisfies
\[
\frac{g}{c^{2^{ \lceil \log_2 n \rceil+1}-1} } \in (f_{e+ \lceil \log_2 n \rceil})
\]
for $e\gg0$, we see that $\frac{g}{c^{2^{ \lceil \log_2 n \rceil+1}-1} }$ and $c$
satisfy the condition in \Cref{thm:criterion}~(2).
Hence $\sht^{\reg}(R/f) \leq \lceil \log_{2} n \rceil +1$.

The proof for type $D_{2n+1}^0$ is the same, by using
\[
g:=x y^{2^{\lceil \log_2 n \rceil+e-1} -n} f^{2^{\lceil \log_2n \rceil +e} -1}
\]
for $f:=z^2 + x^2y + y^n z$ and $c = y^{4n}$.
This completes the proof.
\end{proof}

\begin{table}[!htbp]
\caption{heights of non-taut RDPs}
\centering
\begin{tabular}{|l|c|c|c|c|c|}
\hline
$p$ & type & $f$ & $\sht(R/f)$ &
$\sht^\infty(R/f)$ &
$\sht^{\mathrm{reg}}(R/f)$
\\
\hline
2 & $D_{2n}^{0}$ &$z^2 +x^{2}y+xy^n$
& $\lceil \log_2 n \rceil +1$
& $\lceil \log_2 n \rceil +1$
& $\lceil \log_2 n \rceil +1$
\\
\hline
2 & $D_{2n}^{r}$ &$z^2 +x^{2}y+xy^n+xy^{n-r}z$
& $\lceil \log_2 (n-r) \rceil +1$
& $(*)$
& $(*)$
\\
\hline
2 & $D_{2n+1}^{0}$ &$z^2+x^2y+y^nz$
& $\lceil \log_2 n \rceil +1$
& $\lceil \log_2 n \rceil +1$
& $\lceil \log_2 n \rceil +1$
\\
\hline
2 & $D_{2n+1}^{r}$ &$z^2+x^2y+y^nz+xy^{n-r}z$
& $\lceil \log_2 (n-r) \rceil +1$
& $(*)$
& $(*)$
\\
\hline
2 & $E_{6}^{0}$ &$z^2+x^3+y^2z$
& 2 & 2 & 2 \\
\hline
2 & $E_{6}^{1}$ &$z^2+x^3+y^2z+xyz$
& 1 & 1 & 2 \\
\hline
2 & $E_{7}^{0}$ &$z^2+x^3+xy^3$
& 4 & 4 & 4 \\
\hline
2 & $E_{7}^{1}$ &$z^2+x^3+xy^3+x^2yz$
& 3 & 4 & 4 \\
\hline
2 & $E_{7}^{2}$ &$z^2+x^3+xy^3+y^3z$
& 2 & 3 & 3 \\
\hline
2 & $E_{7}^{3}$ &$z^2+x^3+xy^3+xyz$
& 1 & 1 & 3 \\
\hline
2 & $E_{8}^{0}$ &$z^2+x^3+y^5$
& 4 & 4 & 4 \\
\hline
2 & $E_{8}^{1}$ &$z^2+x^3+y^5+xy^3z$
& 4 & 4 & 4 \\
\hline
2 & $E_{8}^{2}$ &$z^2+x^3+y^5+xy^2z$
& 3 & 4 & 4 \\
\hline
2 & $E_{8}^{3}$ &$z^2+x^3+y^5+y^3z$
& 2 & 4 & 4 \\
\hline
2 & $E_{8}^{4}$ &$z^2+x^3+y^5+xyz$
& 1 & 1 & 3 \\
\hline
3 & $E_{6}^{0}$ & $z^2+x^3+y^4$
& 2 & 2 & 2 \\
\hline
3 & $E_{6}^{1}$ &$z^2+x^3+y^4+x^2y^2$
& 1 & 1 & 2 \\
\hline
3 & $E_{7}^{0}$ & $z^2+x^3+xy^3$
& 2 & 2 & 2 \\
\hline
3 & $E_{7}^{1}$ &$z^2+x^3+xy^3+x^2y^2$
& 1 & 1 & 2 \\
\hline
3 & $E_{8}^{0}$ & $z^2+x^3+y^5$
& 3 & 3 & 3 \\
\hline
3 & $E_{8}^{1}$ & $z^2+x^3+y^5+x^2y^3$
& 2 & 3 & 3 \\
\hline
3 & $E_{8}^{2}$ & $z^2+x^3+y^5+x^2y^2$
& 1 & 1 & 2 \\
\hline
5 & $E_{8}^{0}$ &$z^2+x^3+y^5$
& 2 & 2 & 2 \\
\hline
5 & $E_{8}^{1}$ &$z^2+x^3+y^5+xy^4$
& 1 & 1 & 2 \\
\hline
\end{tabular}
\label{table:RDPs}
\end{table}

\begin{table}[ht]
\caption{}
\centering
\begin{tabular}{|c|c|c|c|c|c|c|c|}
\hline
$p$ & Type & $f$ & $\sht^{\mathrm{reg}}(R/f)$ & $e$ & $a$ & $c$ & $\tau(R/f)$ \\
\hline
2 & $E_6^{0}$ &
$z^2 + x^3 + y^2 z$ &
2 & 5 & $x^3 y$ & $x^4$ & $(x,y,z)$ \\
\hline
2 & $E_6^{1}$ &
$z^2 + x^3 + y^2 z + x y z$ &
2 & 5 & $x^3 y$ & $x^4$ & $(x,y,z)$ \\
\hline
2 & $E_7^{0}$ &
$z^2 + x^3 + x y^3$ &
4 & 7 & $x^{127} y^3 z$ & $y^4$ & $(x,y,z)$ \\
\hline
2 & $E_7^{1}$ &
$z^2 + x^3 + x y^3 + x^2 y z$ &
4 & 6 & $x^3 y^{31} z$ & $x^4$ & $(x,y,z)$ \\
\hline
2 & $E_7^{2}$ &
$z^2 + x^3 + x y^3 + y^3 z$ &
3 & 7 & $x^3 y^{31} z$ & $x^4$ & $(x,y,z)$ \\
\hline
2 & $E_7^{3}$ &
$z^2 + x^3 + x y^3 + x y z$ &
3 & 6 & $x^{16} y^4$ & $y^4$ & $(x,y,z)$ \\
\hline
2 & $E_8^{0}$ &
$z^2 + x^3 + y^5$ &
4 & 7 & $x^3 y^{63} z$ & $x^4$ & $(x,y^2,z)$ \\
\hline
2 & $E_8^{1}$ &
$z^2 + x^3 + y^5 + x y^3 z$ &
4 & 7 & $x^3 y^{63} z$ & $x^4$ & $(x,y^2,z)$ \\
\hline
2 & $E_8^{2}$ &
$z^2 + x^3 + y^5 + x y^2 z$ &
4 & 7 & $x^3 y^{63} z$ & $x^4$ & $(x,y^2,z)$ \\
\hline
2 & $E_8^{3}$ &
$z^2 + x^3 + y^5 + y^3 z$ &
4 & 7 & $x^3 y^{63} z$ & $x^4$ & $(x,y,z)$ \\
\hline
2 & $E_8^{4}$ &
$z^2 + x^3 + y^5 + x y z$ &
3 & 8 & $x^{31} y^3 z$ & $y^4$ & $(x,y,z)$ \\
\hline
3 & $E_6^{0}$ &
$z^2 + x^3 + y^4$ &
2 & 5 & $x^2 y^{48} z^{80}$ & $y^4$ & $(x,y,z)$ \\
\hline
3 & $E_6^{1}$ &
$z^2 + x^3 + y^4 + x^2 y^2$ &
2 & 6 & $x^{83} y^{48} z^{80}$ & $y^4$ & $(x,y,z)$ \\
\hline
3 & $E_7^{0}$ &
$z^2 + x^3 + x y^3$ &
2 & 5 & $y^2 z^{48}$ & $z^4$ & $(x,y,z)$ \\
\hline
3 & $E_7^{1}$ &
$z^2 + x^3 + x y^3 + x^2 y^2$ &
2 & 4 & $y^6 z^8$ & $y^4$ & $(x,y,z)$ \\
\hline
3 & $E_8^{0}$ &
$z^2 + x^3 + y^5$ &
3 & 5 & $x^2 y^{57} z^{80}$ & $y^4$ & $(x,y,z)$ \\
\hline
3 & $E_8^{1}$ &
$z^2 + x^3 + y^5 + x^2 y^3$ &
3 & 4 & $x^8 y^{35} z^8$ & $y^4$ & $(x,y,z)$ \\
\hline
3 & $E_8^{2}$ &
$z^2 + x^3 + y^5 + x^2 y^2$ &
2 & 6 & $x^2 y^{48} z^{80}$ & $y^4$ & $(x,y,z)$ \\
\hline
5 & $E_8^{0}$ &
$z^2 + x^3 + y^5$ &
2 & 4 & $x^{28} y^4 z^{124}$ & $x^4$ & $(x,y,z)$ \\
\hline
5 & $E_8^{1}$ &
$z^2 + x^3 + y^5 + x y^4$ &
2 & 5 & $x^{28} y^4 z^{124}$ & $x^4$ & $(x,y,z)$ \\
\hline
\end{tabular}
\label{table:RDPs-qfr}
\end{table}

\begin{remark}
Liedtke, Martin, and Matsumoto \cite{LMM}*{Proposition~6.2} study the structure of the top local cohomology of Witt rings of rational double points.
Although the quasi-$F^{\infty}$-split height is not computed explicitly in \cite{LMM},
their analysis allows one to determine it after a suitable argument.
More precisely, for a non-$F$-pure RDP $R$, one can deduce from their results that
\[
\sht^{\infty}(R)-1
=
\max\{\, l \in \Z_{\geq 1} \mid \text{$f^{(l)}$ is one of the generators listed in \cite{LMM}*{Table~3}} \,\}.
\]
We emphasize that this computation is obtained by an argument quite different from the one used in this paper.
\end{remark}

\section{On \Cref{question} in the graded case}

In this section, we give an affirmative answer to \Cref{question} for localizations of graded non-$F$-pure normal Gorenstein rings with $F$-rational punctured spectrum.

\begin{theorem}[\Cref{intro:graded}, cf.~\cite{TWY24}*{Corollary~4.19}]\label{graded-case}
Let $S$ be an $F$-finite Noetherian normal $\Z_{\ge 0}$-graded ring of characteristic $p>0$, and set $d := \dim S$.
Assume that $S_0$ is a field, and put $\m := S_+$ and $R := S_\m$.
Assume that $\Spec(R) \setminus \{\m\}$ is $F$-rational and that $R$ is Gorenstein.
If $R$ is not $F$-pure, then $\sht^\infty(R)=\sht^{\reg}(R)$.
\end{theorem}

\begin{proof}
For each integer $n \geq 1$ and each $W_n(R)$-module $M$, we write
\[
H^d_\m(M):=H^d_{W_n(\m)}(M).
\]
For integers $e,n \ge 1$, we set $B^{e}_{R,n} := \Coker(R \to Q^{e}_{R,n})$, and define $B^e_{S,n}$ in the same way.
Since $S$ is graded, the rings $W_n(S)$ and the $W_n(S)$-modules $Q^e_{S,n}$ and $B^e_{S,n}$ carry natural graded structures for all $e,n \geq 1$; see \cite{KTTWYY2}*{Section~7}.
Therefore, $H^i_\m(W_n(R))$, $H^i_\m(Q^e_{R,n})$, and $H^i_\m(B^e_{R,n})$ inherit natural graded structures for all $i,e,n \geq 1$.
We may assume that $h:=\sht^\infty(R) < \infty$.
Set
\[
a(S) := \max\{m \in \Z \mid H^d_\m(S)_m \ne 0\}.
\]

We first show that $a(S) < 0$.
Suppose to the contrary that $a(S) \ge 0$.
From the exact sequence
\[
F_*H^d_\m(W_{n-1}(R)) \xrightarrow{V} H^d_\m(W_n(R)) \xrightarrow{\Res} H^d_\m(R) \to 0,
\]
we obtain $H^d_\m(W_h(R))_m = 0$ for all $m > p^h a(S)$.
In particular, choosing $e \ge 1$ such that $p^e > p^h a(S)$, the homomorphism
\[
H^d_\m(W_h(R))_m \xrightarrow{F^e} F^e_*H^d_\m(W_h(R))_{p^e m}
\]
is zero for all $m \ge 1$.
Since $R$ is $h$-quasi-$F^e$-split, the restriction map
\[
\Res \colon H^d_\m(W_h(R))_m \to H^d_\m(R)_m
\]
is zero for all $m \ge 1$.
As the restriction map map
\[
\Res \colon H^d_\m(W_n(R)) \to H^d_\m(R)
\]
is surjective, we conclude that $H^d_\m(R)_m = 0$ for all $m \ge 1$, and hence $a(S)=0$.

Since $R$ is not $F$-pure, we have $h \ge 2$.
As $\sht^\infty(R)=h$, there exists a positive integer $e$ such that $R$ is not $(h-1)$-quasi-$F^e$-split.
Choose a non-zero element $\eta \in H^d_\m(R)$ contained in the socle; then $\deg(\eta)=0$.

For each positive integer $e'$, we have a commutative diagram with exact rows:
\begin{equation}\label{eq:exact:e'}
\begin{tikzcd}
    H^{d-1}_\m(B^{e'}_{R,h}) \arrow[r,"\alpha^{e'}_{h}"] \arrow[d,"\Res"] &
    H^{d}_\m(R) \arrow[r,"\Phi^{e'}_{R,h}"] \arrow[d,equal] &
    H^{d}_\m(Q^{e'}_{R,h}) \arrow[d] 
    \\
    H^{d-1}_\m(B^{e'}_{R,h-1}) \arrow[r,"\alpha^{e'}_{h-1}"] &
    H^{d}_\m(R) \arrow[r,"\Phi^{e'}_{R,h-1}"] &
    H^{d}_\m(Q^{e'}_{R,h-1})
\end{tikzcd}
\end{equation}

For $e'=e$, since $R$ is not $(h-1)$-quasi-$F^e$-split, there exists a homogeneous element $\tau \in H^{d-1}_\m(B^{e}_{R,h-1})$ of degree $0$ such that $\alpha^e_{h-1}(\tau)=\eta$.

We also have the following commutative diagram with exact rows:
\[
\begin{tikzcd}
    H^{d-1}_\m(B^e_{R,h}) \arrow[r] \arrow[d] &
    H^{d-1}_\m(B^e_{R,h-1}) \arrow[r] \arrow[d,"\beta"] &
    F^{h-1}_*H^{d}_\m(B^e_{R,1}) \arrow[d,"\gamma"]
    \\
    H^{d-1}_\m(B^{e+1}_{R,h}) \arrow[r,"\Res"] &
    H^{d-1}_\m(B^{e+1}_{R,h-1}) \arrow[r] &
    F^{h-1}_*H^{d}_\m(B^{e+1}_{R,1})
\end{tikzcd}
\]

To show that $\beta(\tau) \in \Im(\Res)$, it suffices to prove that $\gamma$ is zero in degree $0$.
Consider the commutative diagram
\begin{equation}\label{eq:exact2}
\begin{tikzcd}
    F^e_*H^d_\m(R) \arrow[r,twoheadrightarrow] \arrow[d,"F"] &
    H^d_\m(B^e_{R,1}) \arrow[d,"\gamma"] \\
    F^{e+1}_*H^d_\m(R) \arrow[r,twoheadrightarrow]  &
    H^d_\m(B^{e+1}_{R,1})
\end{tikzcd}
\end{equation}
Since $R$ is Gorenstein and not $F$-pure, the left vertical map in \eqref{eq:exact2} is zero in degree $0$.
As the horizontal maps are surjective, it follows that $\gamma$ is zero in degree $0$.

Hence there exists $\tau' \in H^{d-1}_\m(B^{e+1}_{R,h})$ such that $\beta(\tau)=\Res(\tau')$.
In particular,
\[
\alpha^{e+1}_h(\tau') = \Res \circ \alpha^{e+1}_{h-1}(\tau) = \eta
\]
by \eqref{eq:exact:e'}.
Thus $\Phi^{e+1}_{R,h}(\tau')=0$, contradicting the fact that $R$ is $h$-quasi-$F^{e+1}$-split.
Therefore, we conclude that $a(S)<0$.

We next show that $\sht^{\infty}(R)=h=\sht^{\reg}(R)$.
Since $h \le \sht^{\reg}(R)$ is clear, it suffices to prove the opposite inequality.
For each $m \ge 1$, set
\[
t_m := \inf\{\,l \in \Z \mid (\widetilde{0^*_m})_l \ne 0\,\}.
\]
Since $a(S)<0$ and $\Spec(R)\setminus\{\m\}$ is $F$-rational, we have $-\infty < t_1 < 0$.
We claim that $t_m = p^{m-1} t_1$ for all $m \ge 1$.
The case $m=1$ is clear.
Assume the claim holds for $m-1$.
By \cite{KTTWYY3}*{Proposition~3.20(2)}, we have
\[
V^{-(m-1)}(\widetilde{0^*_m}) = F^{m-1}_* 0^*,
\]
where
\[
V^{m-1} \colon F^{m-1}_*H^d_\m(R) \longrightarrow H^d_\m(W_m(R)).
\]
Since $R$ is Cohen--Macaulay, the map $V^{m-1}$ is injective, and hence $t_m \ge p^{m-1} t_1$.

Now suppose that $\alpha \in \widetilde{0^*_m}$ is homogeneous of degree $l < p^{m-1}t_1$.
Then, by \cite{KTTWYY3}*{Proposition~3.20},
\[
\Res(\alpha) \in (\widetilde{0^*_{m-1}})_l = 0,
\]
because $t_{m-1} > l$.
Thus there exists $\beta \in H^d_\m(R)$ with $V^{m-1}(\beta)=\alpha$.
Since $\beta$ has degree $l/p^{m-1} < t_1$, we have $\beta=0$, and hence $\alpha=0$.
Therefore, $t_m = p^{m-1} t_1$.

Now suppose that $0^*_h \ne 0$, and choose a homogeneous element $\alpha \in (0^*_h)_s$.
Then $a(S) \ge s$.
Choose $e \ge 1$ such that $p^{h-1} t_1 > p^e a(S)$.
By \cite{KTTWYY3}*{Theorem~3.25}, choose a lift $\alpha_h \in H^d_\m(W_h(R))$ of $\alpha$ with $\alpha_h \in \widetilde{0^*_h}$.
By \cite{KTTWYY3}*{Proposition~3.23}, the element $F^e(\alpha_h)$ is homogeneous of degree $p^e s \le p^e a(S)$.
Since $t_h = p^{h-1} t_1 > p^e a(S)$, we have $F^e(\alpha_h)=0$.
As $R$ is $h$-quasi-$F^e$-split, this implies $\alpha=0$, a contradiction.

Hence $0^*_h = 0$, completing the proof.
\end{proof}

\bibliographystyle{skalpha}
\bibliography{bibliography.bib}
\end{document}